\documentclass[12pt]{amsart}

\usepackage{amsfonts}
\usepackage{amsmath}
\usepackage{amsthm}
\usepackage{url}
\usepackage[all]{xy}
\usepackage{graphicx}
\usepackage{latexsym}
\usepackage{amssymb}
\usepackage[cp850]{inputenc}
\usepackage{amsfonts}

\newtheorem{sat}{Theorem}[section]		
\newtheorem{lem}[sat]{Lemma}
\newtheorem{kor}[sat]{Corollary}			\newtheorem{prop}[sat]{Proposition}
\newtheorem{bei}{Example}				
\newtheorem*{defi*}{Definition}			\newtheorem*{bei*}{Example}
\newtheorem*{sat*}{Theorem}				\newtheorem*{kor*}{Corollary}
\newtheorem*{rmk*}{Remark}				\newtheorem{quest}{Question}	

\let\ssection=\section
\renewcommand{\section}{\setcounter{equation}{0}\ssection}

\newtheorem*{namedtheorem}{\theoremname}
\newcommand{\theoremname}{testing}
\newenvironment{named}[1]{\renewcommand{\theoremname}{#1}\begin{namedtheorem}}{\end{namedtheorem}}

\theoremstyle{remark}
\newtheorem*{bem}{Remark}

\newcommand{\BC}{\mathbb C}			\newcommand{\BH}{\mathbb H}
\newcommand{\BR}{\mathbb R}			
			\newcommand{\BQ}{\mathbb Q}
			\newcommand{\BZ}{\mathbb Z}
				\newcommand{\BT}{\mathbb T}

\newcommand{\CG}{\mathcal G}

		\newcommand{\CN}{\mathcal N}
		\newcommand{\CP}{\mathcal P}

		\newcommand{\CX}{\mathcal X}
		\newcommand{\CZ}{\mathcal Z}

\newcommand{\actson}{\curvearrowright}
\newcommand{\D}{\partial}
\newcommand{\bs}{\backslash}

\DeclareMathOperator{\SL}{SL}		
\DeclareMathOperator{\GL}{GL}		
\DeclareMathOperator{\Id}{Id}		
\DeclareMathOperator{\Isom}{Isom}	

\DeclareMathOperator{\rank}{rank}

\newcommand{\comment}[1]{}

\DeclareMathOperator{\SO}{SO}
\DeclareMathOperator{\Gal}{Gal}

\DeclareMathOperator{\Stab}{Stab}

\DeclareMathOperator{\Span}{Span}

\DeclareMathOperator{\pol}{pol}

\begin{document}

\title[]{Periodic maximal flats are not peripheral}
\author{Alexandra Pettet \& Juan Souto}
\thanks{The first author is partially supported by NSF grant DMS-0856143 and NSF RTG DMS-0602191; the second author was partially supported by NSF grant DMS-0706878 and the Alfred P. Sloan Foundation} 

\begin{abstract}
We prove that every non-positively curved locally symmetric manifold $M$ of finite volume contains a compact set $K$ such that no periodic maximal flat in $M$ can be homotoped out of $K$.
\end{abstract}
\maketitle

\centerline{\em Este art\'iculo est\'a dedicado a Cloe y a la madre que la pari\'o.}

\section{Introduction}
Suppose that $M$ is a non-compact manifold with empty boundary $\D M=\emptyset$. A continuous map $f:X\to M$ is {\em peripheral} if for every compact set $C\subset M$ there is $f_C:X\to M$ homotopic to $f$ and with $f_C(X)\cap C=\emptyset$. If the manifold $M$ is {\em topologically tame}, i.e.~homeomorphic to the interior of a compact manifold with boundary $\bar M$, then $f:X\to M$ is peripheral if and only if it is homotopic, within $\bar M$, to a map whose image is contained in $\D\bar M$. In this note we will be exclusively interested in locally symmetric manifolds $M$ of finite volume; these manifolds are topologically  tame by the work of Borel-Serre \cite{Borel-Serre}. Here, and throughout the paper, we use {\em symmetric space} to mean a simply connected non-positively curved Riemannian symmetric space without euclidean factor. A {\em locally symmetric manifold} $M$ is a Riemannian manifold whose universal cover is a symmetric space. 

The following result is an almost direct consequence of  the work of Borel-Serre \cite{Borel-Serre}:

\begin{sat}\label{peripheral}
If $M$ is an arithmetic locally symmetric manifold then every map $f:X\to M$, where $X$ is a CW-complex of dimension $\dim(X)<\rank_\BQ(M)$, is peripheral.
\end{sat}

There are a number of definitions of arithmetic to be found in the literature; as we will see below, Theorem \ref{peripheral} holds true for any of them. For reasons of expediency, we choose to say that a locally symmetric manifold 
$$M=\Gamma\bs G_\BR^0/K$$
is arithmetic if  $G_\BR^0$ is the identity component of the group of real points of a connected semisimple algebraic group $G\subset\SL_m\BC$ defined over $\BQ$, $K$ is a maximal compact subgroup of $G_\BR^0$, and $\Gamma$ is a torsion free, finite index subgroup of the group $G_\BZ$ of integer points. In the statement of Theorem \ref{peripheral}, $\rank_\BQ(M)$ denotes the $\BQ$-rank of $G$.

It is due to Borel that every arithmetic locally symmetric manifold has finite volume. On the other hand,  the arithmeticity theorem of Margulis asserts that every finite volume locally symmetric manifold which is neither negatively curved nor finitely covered by a product is commensurable to an arithmetic locally symmetric manifold. In section \ref{sec:peripheral} we will deduce that, after giving a reasonable definition of $\rank_\BQ(M)$, Theorem \ref{peripheral} is not only true for arithmetic, but also for arbitrary locally symmetric manifolds of finite volume.

It is well-known that if $M$ has finite volume and is negatively curved, then (parametrized) non-trivial geodesics are not peripheral. In general, this is no longer true. In fact, as an immediate corollary of Theorem \ref{peripheral} we obtain:

\begin{bei}\label{example1}
Suppose that $\Gamma\subset\SL_n\BZ$ is a finite index, torsion free subgroup and consider the locally symmetric manifold $M=\Gamma\bs\SL_n\BR/\SO_n$. Then $\rank_\BQ(M)=n-1$ and hence every geodesic in $M$ is peripheral if $n\ge 3$.
\end{bei}

In order to obtain positive results, in place of geodesics we consider {\em periodic maximal flats} in finite volume locally symmetric manifolds $M$. By a {\em flat} we mean an isometric immersion $\phi:F\to M$ of a complete and connected, locally euclidean manifold $F$ into $M$. The maximal possible dimension $\dim(F)$ of a flat in $M$ is $\rank_\BR(M)$, the real rank of the symmetric space covering $M$. A flat $\phi:F\to M$ with $\dim(F)=\rank_\BR(M)$ is said to be {\em maximal}, and it is {\em periodic} if  $F$ is closed. Existence of periodic maximal flats is due to Prasad-Raghunathan \cite[Theorem 2.8]{Prasad-Raghunathan} in the arithmetic case; for general finite volume locally symmetric manifolds existence follows from the Prasad-Raghunathan result and the arithmeticity theorem of Margulis (see Proposition \ref{per-max-exist} below). We can state now our main result:

\begin{sat}\label{main}
In a finite volume locally symmetric manifold $M$, no periodic maximal flat is peripheral. 
\end{sat}

It follows directly from Theorem \ref{peripheral} that for any $M$ as in Example \ref{example1} there are peripheral periodic flats $f:F\to M$ with $\dim(F)=n-2=\rank_\BR(M)-1$. In fact, it is not difficult to see that this remains true for general non-compact locally symmetric manifolds with finite volume.

\begin{prop}\label{optimal}
For every non-compact locally symmetric manifold $M$ with finite volume there is a peripheral periodic flat $f:F\to M$ with $\dim(F)=\rank_\BR(M)-1$.
\end{prop}

On the other hand, the authors suspect that for every  locally symmetric manifold $M$ with finite volume there is also a non-peripheral periodic flat $f:F\to M$ with $\dim(F)=\rank_\BQ(M)$.

Returning to the statement of Theorem \ref{main}, recall that it is due to Tomanov and Weiss \cite{Tomanov-Weiss} that every locally symmetric manifold $M$ contains a compact set $C$ such that the image of every maximal flat intersects $C$. Theorem \ref{main} asserts that, for periodic maximal flats, this is also true up to homotopy. In fact, if we define a {\em soul} of $M$ to be a compact CW-complex $\CX\subset M$ whose complement $M\setminus\CX$ can be homotoped, within $M$, outside of every compact set $C\subset M$, we have:

\begin{kor}\label{kor:soul}
If $\CX\subset M$ is a soul of $M$, then the image of every map homotopic to a periodic maximal flat intersects $\CX$.\qed
\end{kor}

Being topologically tame, every finite volume locally symmetric manifold $M$ has a soul. For instance, if $M$ is negatively curved, a suitably chosen thick part is a soul. 

In the case that $M=\Gamma\bs\SL_n\BR/\SO_n$ with $\Gamma\subset\SL_n\BZ$, a particularly nice soul can be described. To begin with, identify $\SL_n\BR/\SO_n$ with the space of isometry classes of marked lattices in $\BR^n$ with covolume $1$. A lattice is {\em well-rounded} if the set of its shortest non-trivial vectors spans $\BR^n$ as an $\BR$-vector space. The subset $\CX\subset\SL_n\BR/\SO_n$ consisting of isometry classes of well-rounded lattices is the so called {\em well-rounded retract}. It is due to Soul\'e \cite{Soule} for $n=3$ and Ash \cite{Ash} in general that $\CX$ is an $\SL_n\BZ$-invariant deformation retract of $\SL_n\BR/\SO_n$. If $\Gamma$ is a finite index, torsion free subgroup of $\SL_n\BZ$ we denote also by $\CX$ the projection of $\CX\subset\SL_n\BR/\SO_n$ to $M=\Gamma\bs\SL_n\BR/\SO_n$ and refer to it as the well-rounded retract of $M$. It follows from \cite[Proposition 6.3]{wr-minimal} that the well-rounded retract is a soul of $M$. In particular, we deduce that the image of every map homotopic to a periodic maximal flat in $M$ intersects the well-rounded retract. As a corollary we obtain a proof of the following result due to McMullen  \cite[Theorem 4.1]{McMullen}:

\begin{sat}[McMullen]\label{McMullen-wr}
Let $H\subset\SL_n\BR$ be the diagonal group and suppose that $A\in\SL_n\BR$ is such that the projection of $HA$ is compact in $\SL_n\BR/\SL_n\BZ$. Then there is $B\in H$ such that $BA\BZ^n$ is a well-rounded lattice.
\end{sat}

Claiming that we obtain a new proof of McMullen's theorem is somewhat misleading: the proof of Theorem \ref{main} follows McMullen's argument closely, and we rely heavily on some of his tools. In some sense, the result of this paper is to confirm in yet another instance that  whenever something is true for $\SL_n\BZ\bs\SL_n\BR/\SO_n$, then it is also true in some form for general locally symmetric spaces of finite volume. 

\begin{bem}
For the sake of simplicity, throughout this note we consider only manifolds, or equivalently torsion-free lattices. The statements in the present note remain true for general finite volume locally symmetric orbifolds if we replace homotopy by what could be poetically baptized as {\em homotopy in the sense of orbifolds} or {\em orbifold homotopy}. For instance, this can be easily seen by passing to a finite manifold cover. 
\end{bem}

We now sketch briefly the proof of Theorem \ref{main}. Let $M$ be an arithmetic locally symmetric manifold and denote by $\pi:S\to M$ its universal cover; $S$ is a symmetric space. It follows from the work of Borel-Serre that $M$ contains a compact set $C$ with the property that $S\setminus\pi^{-1}(C)$ admits a $\pi_1(M)$-equivariant map
$$\tau:S\setminus\pi^{-1}(C)\to\Delta_\BQ(G)$$
 to the rational Tits building $\Delta_\BQ(G)$ of the algebraic group corresponding to $M$. Suppose that $\phi:F\to M$ is homotopic to a periodic maximal flat and that $\phi(F)\cap C=\emptyset$; without loss of generality we may assume that $F$ is a topological torus. Let $\tilde F$ be the universal cover of $F$ and $\tilde\phi:\tilde F\to S\setminus\pi^{-1}(C)$ a lift of $\phi$. For $k=0,\dots,\rank_\BQ(G)-1$ let $\tilde U_k\subset\tilde F$ be the pre-image under the map $\tau\circ\tilde\phi$ of the union of the open $k$-simplices in $\Delta_\BQ(G)$. By construction $\tilde U_k$ is $\pi_1(F)$-invariant; denote by $U_k$ its projection to $F$ and notice that the $U_k$'s cover $F$. Using basic properties of algebraic groups we will show that the rank, as an abelian group, of the image of the homomorphism $H_1(V,\BZ)\to H_1(F,\BZ)$  is bounded from above by $\dim(F)-k-1$ for every connected component $V$ of $U_k$. We deduce from a beautiful result due to McMullen \cite[Theorem 2.1]{McMullen} that such a covering of the torus $F$ cannot exist; this contradicts the assumption that $\phi:F\to M$ is homotopic to a periodic maximal flat and that $\phi(F)\cap C=\emptyset$.
\medskip

We conclude this introduction with a brief plan of the paper. In sections \ref{sec:gkhk} and \ref{sec:lss} we recall some facts on algebraic groups and symmetric spaces. In section \ref{sec:peripheral} we prove Theorem \ref{peripheral}; this result is surely well-known but we include a proof for completeness. After some preparatory work in section \ref{sec:lemmas}, we prove Theorem \ref{main} and Proposition \ref{optimal} in section \ref{sec:proof1}. Finally, in section \ref{sec:final} we derive McMullen's theorem from Theorem \ref{main} and add a few remarks which are hopefully of interest to the reader.
\medskip

Both authors of this note are far from being experts in the geometry of locally symmetric spaces and hence it is likely that some of our arguments are unnecessarily cumbersome; we apologize to the reader in advance. 

Many thanks are due to Mladen Bestvina, Alex Eskin, Philippe Eyssidieux, Benson Farb, Lizhen Ji, Gopal Prasad, Ralf Spatzier and Domingo Toledo for listening to and answering our questions with a patience worthy of Job \cite[5:11]{bible}. To a large extent, this paper was written while the authors were visiting  the University of Toronto and Peking University; we are grateful for the hospitality of these institutions. Finally, we would like to thank Joan Knoertzer, without whom it would have been much more difficult for us to spend time in Toronto and Beijing.

\section{Gut kopiert, halb kapiert}\label{sec:gkhk}
In this section we recall a few definitions and basic facts on algebraic groups. We refer the reader to \cite{Borel-annals,Borel-arit,Borel20,Borel-ag,Borel-li,BHC,BT,Brown,Tits} for details. The authors found the brief survey \cite{Borel20} and the beautiful book \cite{Borel-arit} particularly useful. Lacking sufficient knowledge, we will not aim to cite original sources.

\subsection{Algebraic groups}

In the sequel, $k$ will be some subfield of the complex numbers $\BC$; we will be mostly interested in the cases $k=\BQ,\BR$. Throughout this section, topological notions refer to the Zariski topology. In order to make this explicit we will write Z-closed, Z-connected, etc...

We understand an {\em algebraic group defined over $k$} to be a subgroup $G\subset\SL_m\BC$ which is determined by a collection of polynomial equations with coefficients in $k$ \cite{Borel-ag}. Since $k$ has characteristic $0$, this last condition is equivalent to $G$ being an algebraic variety defined over $k$. If $R$ is a subring of $\BC$ then we denote by $G_R$ the set of points in $G=G_\BC$ with coefficients in $R$. We will be mainly interested in the cases $R=\BZ,\BQ,\BR$. Recall that the set of real points $G_\BR$ of an algebraic group defined over $\BR$ is a Lie group with finitety many connected components; the group $G$, i.e.~the set of complex points, is Z-connected if and only if it is connected in the analytic topology \cite[1.5]{Borel20}.

Observe that an algebraic group $G$ is defined over $k$ if and only if it is invariant, as a set, under the action of the Galois group $\Gal(\BC/k)$ on $\SL_m\BC$ \cite[AG.14.4]{Borel-ag}. In particular, if $G$ is defined over $k$ and $H\subset G_k$ is an arbitrary subset, then the {\em centralizer} and the {\em normalizer} 
$$\CZ_G(H)=\{g\in G\vert gh=hg\ \forall h\in H\}\ \ \ \CN_G(H)=\{g\in G\vert gH=Hg\}$$
of $H$ in $G$ are algebraic groups defined over $k$. Similarly, if $\bar H$ is the subgroup of $G$ generated by $H$, and $\Span(H)$ is the connected component of the identity of the Z-closure of $\bar H$, then $\Span(H)$ is an algebraic group defined over $k$. 

An element $g\in\GL_n\BC$ is {\em semisimple} if it is diagonalizable over $\BC$; $g$ is {\em unipotent} if $(g-\Id)^n=0$. An algebraic group $G$ is {\em unipotent} if it consists of unipotent elements. Every unipotent group is connected and conjugate to a subgroup consisting of upper triangular matrices with ones in the diagonal (see section 4 in \cite{Borel20}). The latter fact implies that every unipotent group is nilpotent as an abstract group; however the converse is not true. An algebraic group $G$ is {\em solvable} if it is solvable as an abstract group. As long as $G$ is connected, $G$ is solvable if and only if it is conjugate to a subgroup of upper triangular matrices \cite[III.10.5]{Borel-ag}. 

The {\em radical} (resp., {\em unipotent radical}) $R(G)$ (resp., $R_u(G)$) of an algebraic group $G$ is the largest connected normal solvable  (unipotent) algebraic subgroup of $G$. If $G$ is defined over $k$ then both $R(G)$ and $R_u(G)$ are also defined over $k$ \cite[0.7]{BT}. The group $G$ is {\em semisimple} (resp., {\em reductive}) if $R(G)=\Id$ (resp., $R_u(G)=\Id$).

A reductive subgroup $H$ of an algebraic group $G$ is a {\em Levi subgroup} if $G$ is the semi-direct product
$$G=H\ltimes R_u(G)$$
of its unipotent radical and $H$. Recall that this means that $R_u(G)$ is normal, $G$ is generated by $H$ and $R_u(G)$, and $H\cap R_u(G)=\Id$.

Since we are working over fields of characteristic $0$, every algebraic group $G$ defined over $k$ contains a Levi-subgroup $H$ defined over $k$; moreover, every reductive $k$-subgroup $H'\subset G$ can be conjugated into $H$ by an element in $R_u(G)_k$ \cite[0.8]{BT}. Observe that $H$ is connected if and only if $G$ is connected.

\subsection{Tori}
A $d$-dimensional {\em algebraic torus}, or simply a torus, is an algebraic group $T$ isomorphic over $\BC$ to the direct product of $d$ copies of $\GL_1\BC\simeq\BC^*$. Equivalently $T$ is a connected, algebraic group consisting of semisimple elements \cite[1.1]{BT}. Before moving on, recall that if $G$ is a reductive algebraic group and $T\subset G$ is a torus, then the centralizer $\CZ_G(T)$ of $T$ in $G$ is reductive \cite[2.15 d)]{BT}.

A torus $T$ is {\em $k$-split} if it is defined over $k$ and isomorphic over $k$ to a product of $\GL_1(\BC)$'s. Equivalently, $T$ is defined and diagonalizable over $k$. We summarize a few properties about $k$-split tori that we will need in the sequel:

\begin{prop}\label{prop:split}
(1) Every connected algebraic subgroup of a $k$-split torus is a $k$-split torus as well.
(2) If $A\subset\GL_mk$ is a collection of commuting elements which are diagonalizable over $k$, then $\Span(A)$ is a $k$-split torus. 
\end{prop}

Recall that $\Span(A)$ is the identity component of the Zariski closure of the group generated by $A$. See \cite[1.6]{BT} for the first claim of Proposition \ref{prop:split}. The second claim follows from the fact that commuting semisimple elements can be simultanously diagonalized \cite[I.4.6]{Borel-ag}.

Suppose now that $G$ is a connected reductive algebraic group defined over $k$. A $k$-split torus in $G$ is {\em maximal} if it is not properly contained in any other $k$-split torus. It is known that any two maximal $k$-split tori in $G$ are conjugate in $G$ \cite[4.21]{BT} and hence have the same dimension. The {\em $k$-rank} of $G$, denoted by $\rank_k(G)$, is the dimension of some, and hence every, maximal $k$-split torus in $G$.

\subsection{Parabolic subgroups and the Tits building of $G$}
Suppose now that $G$ is a connected reductive algebraic group defined over $k$ and recall that a {\em Borel subgroup} $B$ of $G$ is a maximal connected solvable algebraic subgroup. Borel subgroups exist and are all conjugate to each other as subgroups of $G$. A proper Z-closed subgroup $P\subset G$ is {\em parabolic} if it contains some Borel subgroup. Equivalently, the algebraic variety $G/P$ is projective \cite[Section 16]{Borel-annals}. Every parabolic subgroup is connected \cite[4.2]{BT}. Observe that no Borel subgroup of $G$ needs to be defined over $k$; in fact $G$ contains a parabolic subgroup defined over $k$ if and only if it contains a non-central $k$-split torus \cite[4.17]{BT}.

The following fact will play a central role in our arguments below.

\begin{sat}\label{parabolic-levy}
Suppose that $P\subset G$ is a parabolic $k$-subgroup. Then there is a $k$-split torus $S\subset P$ whose centralizer $\CZ_G(S)$ in $G$ is a Levi-subgroup of $P$.
\end{sat}

Let $\CP$ be the set of all (proper) maximal parabolic $k$-subgroups. The {\em $k$-Tits building} associated to $G$ is the simplicial complex $\Delta_k(G)$ whose vertices are the elements in $\CP$ where $\{P_0,\dots,P_r\}\subset\CP$ determines a $r$-simplex if $P_0\cap\dots\cap P_r$ is parabolic. If $k=\BQ$, we will refer to $\Delta_\BQ(G)$ as the {\em rational Tits building} of $G$.

\begin{sat}\label{facts-tits}
The Tits building $\Delta_k(G)$ of $G$ is a spherical building of dimension $\rank_k(G)-1$. The action by conjugation of $G_k$ on the set of maximal parabolic subgroups induces a simplicial action on $\Delta_k(G)$. The stabilizer of an $r$-simplex $\{P_0,\dots,P_r\}$ in $\Delta_k(G)$ is the intersection of $G_k$ with the parabolic subgroup 
$$P=P_0\cap\dots\cap P_r$$ 
Moreover, the $k$-split torus $S\subset P$ provided by Theorem \ref{parabolic-levy} has dimension $r+1$.
\end{sat}

For facts on parabolic subgroups see sections 4 and 5 in \cite{BT}. For instance, Theorem \ref{facts-tits} follows from the fact that every parabolic $k$-subgroup is (1) equal to its normalizer \cite[4.4 a]{BT} and (2) conjugate by an element in $G_k$ to a unique {\em standard parabolic subgroup} \cite[5.14]{BT}. See \cite[4.15 b]{BT} for a much more precise statement of Theorem \ref{parabolic-levy}.

\begin{bem}
To conclude we would like to observe that Theorem \ref{facts-tits} fails if $G$ is not connected. In fact, if $G$ is the group generated by $\SL_2\BR\times\SL_2\BR$ and the involution mapping the first and second vectors of the standard basis of $\BR^4$ to the third and fourth, then the normalizer of a Borel subgroup $B$ of $G$ properly contains $B$. It is for this reason that we will define {\em arithmetic} in a rather restrictive way below. This is simply a question of terminology: as we will see below, the results in this paper hold for all locally symmetric spaces, arithmetic or not.
\end{bem}

\section{Locally symmetric spaces}\label{sec:lss}
In this section we review some facts about symmetric spaces, locally symmetric manifolds, and arithmetic groups. See \cite{Borel-arit,Eberlein,Ji-li,Mostow,Witte-Morris} for more on these topics.

\subsection{Symmetric spaces}\label{sec-bla1}
Let $S$ be a (simply connected) Riemannian symmetric space of non-positive curvature. The symmetric space $S$ is {\em irreducible} if it is not isometric to the Riemannian product of two lower-dimensional symmetric spaces; otherwise it is {\em reducible}. It is due to de Rham that every symmetric space $S$ admits a canonical decomposition
\begin{equation}\label{de Rham}
S=S_1\times\dots\times S_s\times\BR^k
\end{equation}
as a product of irreducible symmetric spaces and a euclidean factor. 

The identity component $\Isom(S)^0$ of the isometry group $\Isom(S)$ preserves not only the factors of the de Rham decomposition \eqref{de Rham} but also their order. Recall that $\Isom(S)^0$ has finite index in $\Isom(S)$.

\begin{quote}{\bf Terminology:} By a {\em symmetric space} we mean from now on a symmetric space of non-positive curvature without euclidean factor in its de Rham decomposition.
\end{quote}

An irreducible symmetric space $S$ admits a unique $\Isom(S)^0$-invariant metric up to scaling. However, if $S$ is reducible then the factors in the de Rham decomposition can be scaled by different positive real numbers. Since all these metrics share the same basic properties we will say that they are {\em equivalent} and will not distinguish between them.

It is well-known that for every symmetric space $S$ there is some connected semisimple algebraic group $G$ defined over $\BR$ such that 
\begin{equation}\label{sym-g/k}
S=G_\BR^0/K
\end{equation}
where $K$ is a maximal compact subgroup of $G_\BR^0$, the identity component of the Lie group of real points of $G$. Under the identification \eqref{sym-g/k} the left action of $G_\BR^0$ on $S$ is by isometries. 

If $G'$ is a second semisimple algebraic group and $K'\subset {G'}_\BR^0$ a maximal compact subgroup such that $S={G'}_\BR^0/K'$, then both groups $G$ and $G'$ have the same $\BR$-rank. The {\em rank} of the symmetric space $S$ is then by definition
$$\rank_\BR(S)=\rank_\BR(G)$$
Recall that the symmetric space $S$ has $\rank_\BR(S)=1$ if and only if it is negatively curved.

\subsection{Maximal flats and $\BR$-split tori}
A {\em flat} in the symmetric space $S=G_\BR^0/K$ is a totally geodesic complete locally euclidean submanifold. Since $S$ has non-positive curvature, every flat is simply connected and hence isometric to the euclidean space of the same dimension. A flat is {\em maximal} if it is not properly contained in any other flat. It is known that any two maximal flats are translates of each other under the isometric action $G_\BR^0\actson S$. Their common dimension coincides with the rank of the symmetric space.

Before moving on to more interesting topics we remind the reader of the following fact:

\begin{prop}\label{flat-product}
Let $S$ be a symmetric space with de Rham decomposition $S=S_1\times\dots\times S_s$. Then for every maximal flat $F\subset S$ there are maximal flats $F_1\subset S_1,\dots,F_s\subset S_s$ with $F=F_1\times\dots\times F_s$. Conversely, the product of maximal flats in $S_1,\dots,S_s$ is a maximal flat in $S$.
\end{prop}

Continuing with the same notation, suppose that $A\subset G$ is a maximal $\BR$-split torus. The action of the group $A_\BR^0$ on $S=G_\BR^0/K$ preserves a maximal flat $F\subset S$. In fact, the actions of $A_\BR^0$ on $F$ and of $\BR^{\rank_\BR(S)}$ on itself by translations are conjugate by an isometry $F\to\BR^{\rank_\BR(S)}$. Conversely, every maximal flat arises in this way. More precisely we have:

\begin{prop}\label{R-split-flat}
Let $G$ be a semisimple algebraic group defined over $\BR$, $K\subset G_\BR^0$ a maximal compact subgroup of the group of real points of $G$, and $F$ a maximal flat of $S=G_\BR^0/K$. Then there is a unique maximal $\BR$-split torus $A$ such that $A_\BR^0$ acts simply transitively on $F$. Moreover, this map from the set of maximal flat submanifolds to the set of maximal $\BR$-split tori is a bijection.
\end{prop}

Given a maximal flat $F\subset S=G_\BR^0/K$ and the associated maximal $\BR$-split torus $A\subset G$ denote by $\CZ_{G_\BR^0}(A_\BR^0)$ and $\CN_{G_\BR^0}(A_\BR^0)$ the centralizer and normalizer of $A_\BR^0$ in $G_\BR^0$. Let also 
$$\Stab_{G_\BR^0}(F)=\{g\in G_\BR^0\vert gF=F\}$$
be the stabilizer of $F$ under the action $G_\BR^0\actson S$. We remind the reader of the relation between these groups:

\begin{prop}\label{split-flat}
Let $G$ be a semisimple algebraic group defined over $\BR$, $K\subset G_\BR^0$ a maximal compact subgroup of the identity component of the group of real points of $G$, and $F$ a maximal flat of $S=G_\BR^0/K$ with associated maximal $\BR$-split torus $A$. Then
\begin{itemize} 
\item $\Stab_{G_\BR^0}(F)=\CN_{G_\BR^0}(A_\BR^0)$.
\item $\CZ_{G_\BR^0}(A_\BR^0)$ has finite index in $\CN_{G_\BR^0}(A_\BR^0)$.
\item Every element in $\CZ_{G_\BR^0}(A_\BR^0)$ is semisimple.
\item $\CZ_{G_\BR^0}(A_\BR^0)$ is the direct product of a compact group and $A_\BR^0$. The projection
\begin{equation}\label{polproj}
\pol:\CZ_{G_\BR^0}(A_\BR^0)\to A_\BR^0
\end{equation}
given by this splitting associates to every $g\in\CZ_{G_\BR^0}(A_\BR^0)$ its polar part $\pol(g)$.
\item The action of $\CZ_{G_\BR^0}(A_\BR^0)$ on $F$ factors through the projection \eqref{polproj}.
\item If $H\subset G$ is an algebraic subgroup defined over $\BR$ containing $g\in\CZ_{G_\BR^0}(A_\BR^0)$ then $H$ also contains the image of $g$ under the projection \eqref{polproj}.
\end{itemize}
\end{prop} 

All the statements of Proposition \ref{split-flat} can be found for instance in \cite{Mostow}; see also section 1 in \cite{Prasad-Raghunathan} for a proof of the last claim.

\subsection{Locally symmetric manifolds}
Let $S$ be a symmetric space and $\Gamma$ a discrete subgroup of the isometry group $\Isom(S)$ with associated quotient $\Gamma\bs S$. If the quotient has finite volume then $\Gamma$ is a {\em lattice}. A cocompact lattice is {\em uniform}; otherwise it is {\em non-uniform}. We will be mainly interested in torsion free non-uniform lattices. Observe that if $\Gamma$ is a torsion free lattice then the quotient $M=\Gamma\bs S$ is a finite volume Riemannian manifold and the projection $\pi:S\to M$ is a covering map. We refer to any manifold which arises in this way as a {\em locally symmetric manifold}. Observe that according to this convention every locally symmetric manifold automatically has finite volume.

A locally symmetric manifold $M$ is {\em reducible} if there are locally symmetric manifolds $M_1,\dots,M_r$ and a finite cover
$$M_1\times\dots\times M_r\to M$$
Otherwise $M$ is {\em irreducible}. Observe that every locally symmetric manifold is finitely covered by a product of irreducible locally symmetric manifolds. Two locally symmetric manifolds $M$ and $N$ are {\em commensurable} if they have finite covers $M'\to M$ and $N'\to N$ which, up to replacing the metric by an equivalent metric, are isometric $M'\simeq N'$; see section \ref{sec-bla1} above. 

Abusing terminology, we will say that a locally symmetric manifold is, say blue, if its universal cover is. For instance, the {\em $\BR$-rank} of a locally symmetric manifold $M=\Gamma\bs S$ is the rank of $S$: $\rank_\BR(M)=\rank_\BR(S)$.

\subsection{Arithmetic manifolds}
Let $G$ be a connected semisimple algebraic group defined over $\BQ$. We say that a finite index subgroup $\Gamma\subset G_\BZ\cap G_\BR^0$ is {\em arithmetic}. If $\Gamma$ is arithmetic and torsion free, and $K\subset G_\BR^0$ is a maximal compact subgroup of the identity component of the group of real points of $G$, then we say that the associated locally symmetric manifold $M=\Gamma\bs G_\BR^0/K$ is {\em arithmetic} as well. It is due to Borel \cite{Borel-arit} that every arithmetic locally symmetric manifold has finite volume. 

Before going any further, recall that two algebraic groups admitting equivalent arithmetic quotients have the same $\BQ$-rank. Hence, it is unambiguous to define the $\BQ$-rank of an arithmetic locally symmetric manifold $M$ as the $\BQ$-rank of the corresponding algebraic group $G$
$$\rank_\BQ(M)=\rank_\BQ(G)$$
It is due to Godement that the compactness of $M$ is equivalent to the vanishing of its $\BQ$-rank (see \cite{Borel-arit}).

\begin{named}{Compactness criterium}[Godement]
Suppose that $G$ is a connected semisimple algebraic group defined over $\BQ$. Then $G_\BZ\bs G_\BR^0$ is compact if and only if $\rank_\BQ(G)=0$.
\end{named}

We now give a more precise statement in the case that $M$ is not compact.

\begin{prop}\label{q-split-proper}
Let $\Gamma$ be an arithmetic subgroup of a connected semisimple algebraic group $G$ defined over $\BQ$ and let $A\subset G$ be a $\BQ$-split torus. Then for all $g\in G_\BR^0$, the map
$$A_\BR^0\to \Gamma\bs G_\BR^0,\ \ \ h\mapsto[hg]$$ 
is proper. Here $[x]$ is the class of $x\in G_\BR^0$ in $\Gamma\bs G_\BR^0$.
\end{prop}

From a geometric point of view, a much more precise and statement can be found in \cite{Ji-MacPherson} (see also \cite{Leuzinger}). For an again much better version of Proposition \ref{q-split-proper}, this time from a dynamical point of view, see \cite{Tomanov-Weiss}.

Arithmetic locally symmetric manifolds play a role in the present paper because the arithmeticity theorem of Margulis \cite{Margulis} essentially asserts that every locally symmetric manifold which is neither a product nor negatively curved is arithmetic.

\begin{named}{Arithmeticity theorem}[Margulis]
Every irreducible, locally symmetric manifold $M$ with finite volume and $\rank_\BR(M)\ge 2$ is commensurable to an arithmetic locally symmetric manifold.
\end{named}

For further reference we state here the following consequence of the arithmeticity theorem:

\begin{kor}\label{kor-arit}
Every finite volume locally symmetric manifold $M$ is finitely covered by a product of negatively curved finite volume locally symmetric spaces and arithmetic locally symmetric spaces.
\end{kor}

\begin{bem}
In the literature, a locally symmetric manifold which is commensurable with an arithmetic manifold is often said to be itself arithmetic. In order to keep track of the connectivity properties of the associated algebraic group we have decided to give the more restrictive definition above. This does not reduce the scope in which our theorems are valid: the staments of Theorem \ref{peripheral} and Theorem \ref{main} are true for a locally symmetric manifold if and only if they are true for some, and hence for every, finite cover.
\end{bem}

\subsection{Ends of locally symmetric manifolds}
The ends of finite volume negatively curved locally symmetric manifolds, the so-called {\em cusps}, are well understood. Any such manifold $M$ contains a compact submanifold $\bar M\subset M$ whose complement is homeomorphic to $\D\bar M\times\BR$. Moreover, if $U$ is the closure of a connected component of the pre-image of $M\setminus\bar M$ under the covering map $\pi:S\to M$ then the group $\{\gamma\in\Gamma\vert\gamma U=U\}$ contains a finite index subgroup consisting of unipotent elements, and $\D U$ is homeomorphic to euclidean space and is hence contractible. We can restate this as follows:

\begin{prop}\label{thin-thick}
Let $S$ be a symmetric space with $\rank_\BR(S)=1$ and $\Gamma\subset\Isom(S)$ a torsion free lattice. The symmetric space $S$ admits a $\Gamma$-equivariant bordification $\bar S$ with the following properties:
\begin{enumerate}
\item The action $\Gamma\actson\bar S$ is properly discontinuous and free; the manifold $\Gamma\bs\bar S$ is compact, and its interior is homeomorphic to $\Gamma\bs S$.
\item $\bar S$ is contractible, and $\D\bar S$ is homotopy equivalent to a discrete set of points.
\item If $Z\subset\D\bar S$ is a connected component, then $\Stab_\Gamma(Z)$ contains a finite index subgroup consisting of unipotent elements. 
\end{enumerate}
\end{prop}

Recall that a bordification $\bar V$ of a manifold $V$ is a manifold with boundary such that $V$ is homeomorphic to $\bar V\setminus\D\bar V$.

In the setting of arithmetic lattices, a suitable bordification of the symmetric space $S$ was constructed by Borel-Serre \cite{Borel-Serre}; related constructions are due to Grayson and Leuzinger \cite{Grayson,Leuzinger-BS}. We state here the results of \cite{Borel-Serre} as we will need them here:

\begin{sat}[Borel-Serre \cite{Borel-Serre}]\label{Borel-Serre}
Let $G$ be a connected semisimple algebraic group defined over $\BQ$, $G_\BR^0$ the connected component of the identity of the group of real points, $K\subset G_\BR^0$ a maximal compact subgroup and $S=G_\BR^0/K$ the associated symmetric space. 

There is a $G_\BQ$-equivariant bordification $\bar S$ of $S$ such that whenever $\Gamma\subset G_\BZ$ is a torsion free arithmetic subgroup of $G$ then:
\begin{enumerate}
\item The action $\Gamma\actson\bar S$ is properly discontinuous and free; the manifold $\Gamma\bs \bar S$ is compact, and its interior is homeomorphic to $\Gamma\bs\bar S$.
\item $\bar S$ is contractible and there is a $\Gamma$-equivariant homotopy equivalence $\tau:\D\bar S\to\Delta_\BQ(G)$ from the boundary of $\bar S$ to the rational Tits building $\Delta_\BQ(G)$ associated to $G$.
\end{enumerate}
\end{sat}

The statement that the homotopy equivalence $\tau$ can be chosen to be $\Gamma$-equivariant is not explicitly proved in  \cite{Borel-Serre}; it follows from the fact that $\Gamma$ acts on $\D\bar S$ freely and discretely.

\section{Small dimensional manifolds are peripheral}\label{sec:peripheral}
In this section we prove Theorem \ref{peripheral}.

\begin{named}{Theorem \ref{peripheral}}
If $M=\Gamma\bs S$ is an arithmetic locally symmetric manifold then every map $f:X\to M$, where $X$ is a CW-complex of dimension $\dim(X)<\rank_\BQ(M)$, is peripheral.
\end{named}

Let $\bar S$ be the bordification of $S$ provided by Theorem \ref{Borel-Serre} and recall that $M$ is homeomorphic to the interior of $\bar M=\Gamma\bs\bar S$. In order to prove that the map $f:X\to M$ is peripheral, it suffices to prove that it is homotopic, within $\bar M$, to a map $X\to\D\bar M$. We will prove that this is the case by induction on the dimension of $X$. 

If $\dim(X)=0$ then $X$ is a discrete set of points. Since $\rank_\BQ(M)\ge 1$, it follows from Godement's compactness criterium that $M$ is non-compact and hence $\D\bar M$ is not empty. We can clearly homotope $f$ so that its image is contained in $\D\bar M$.

Suppose now that $\dim(X)=d\ge 1$ and that we have proved Theorem \ref{peripheral} for those CW-complexes with at most dimension $d-1$. We can in particular apply Theorem \ref{peripheral} to the restriction of $f$ to the $(d-1)$-skeleton of $X$. In other words, we can assume that the image under $f$ of the $(d-1)$-skeleton of $X$ is contained in $\D\bar M$. 

Consider the map $\tilde f:\tilde X\to\bar S$ between the universal coverings of $X$ and $\bar M$ respectively and let $\Delta_1,\Delta_2,\dots$ be representatives in $\tilde X$ of the classes of closed $d$-dimensional cells under the action of $\pi_1(X)$. In order to prove that $f$ is homotopic to a map with image in $\D\bar M$ it suffices now to show that for each $i$, the map $\tilde f\vert_{\Delta_i}:\Delta_i\to\bar S$ is properly homotopic to a map with image in $\D\bar S$. 

To prove that this is the case recall that $\D\bar S$ is homotopy equivalent to the rational Tits building $\Delta_\BQ(G)$. The latter is is homotopically equivalent to a bouquet of $(\rank_\BQ(M)-1)$-dimensional spheres \cite[Theorem 8.5.1]{Borel-Serre}; in particular, the assumption that 
$$\rank_\BQ(M)>d\ge 1$$
implies that $\pi_{d-1}(\D\bar S)=0$. Hence, there is a map
$$f_i':\Delta_i\to\D\bar S$$
with $f_i'\vert_{\D\Delta_i}=\tilde f\vert_{\D\Delta_i}$. Since $\bar S$ is itself contractible we deduce that $f_i'$ and $\tilde f\vert_{\Delta_i}$ are homotopic. This concludes the induction step and thus the proof of Theorem \ref{peripheral}.\qed
\medskip

Before moving on to more interesting topics we make some remarks on the general situation. We have to define first the $\BQ$-rank of an arbitrary finite volume locally symmetric manifold $M$:
\begin{itemize}
\item If $M$ is commensurable to an arithmetic locally symmetric manifold $M'$ we set $\rank_\BQ(M)=\rank_\BQ(M')$.
\item If $\rank_\BR(M)=1$, we set $\rank_\BQ(M)=0$ if $M$ is compact, and $\rank_\BQ(M)=1$ otherwise. 
\item If $M$ is covered by a product $M_1\times\dots\times M_s\to M$ we define 
$$\rank_\BQ(M)=\rank_\BQ(M_1)+\dots+\rank_\BQ(M_s)$$
\end{itemize}
These conventions are sometimes redundant, but they are always consistent. Most importantly, it follows from the arithmeticity theorem (see Corollary \ref{kor-arit}) that we have now defined the $\BQ$-rank of an arbitrary finite volume locally symmetric manifold.

Suppose now that $M$ is a general finite volume locally symmetric manifold and let $M_1,\dots,M_s$ be finite volume locally symmetric manifolds such that there is a finite cover
$$M_1\times\dots\times M_s\to M$$
and such that $M_i$ is either negatively curved or arithmetic for each $i$. Denoting by $S$ the universal cover of $M$ and by $S_i$ the universal cover of $M_i$, consider for $i=1,\dots,r$ the bordification $\bar S_i$ provided by Proposition \ref{thin-thick} and Theorem \ref{Borel-Serre} and set
$$\bar S=\bar S_1\times\dots\times\bar S_s$$
The manifold $\bar S$ is a bordification of $S$, and its boundary $\D\bar S$ is homotopy equivalent to a bouquet of $(\rank_\BQ(M)-1)$-spheres. Once this is said, the same argument used to prove Theorem \ref{peripheral} shows that every map
$$f:X\to M$$
with $\dim(X)<\rank_\BQ(M)$ is peripheral. The details are left to the reader.

\section{Periodic maximal flats}\label{sec:lemmas}
In this section we recall a few facts about periodic maximal flats. None of the statements proved here is going to surprise any half-expert in the field.
\medskip

Recall that a periodic maximal flat in a locally symmetric manifold $M=\Gamma\bs S$ is an isometric immersion
$$\phi:F\to M$$
of a compact locally euclidean Riemannian manifold of dimension $\dim F=\rank_\BR(M)$. The following result, essentially due to Prasad and Raghunathan, asserts that periodic maximal flats exist.

\begin{prop}\label{per-max-exist}
Every finite volume locally symmetric manifold contains a periodic maximal flat.
\end{prop}
\begin{proof}
Observe that if a finite cover $M'$ of the locally symmetric space $M$ has a periodic maximal flat, then so does $M$. Hence, by Corollary \ref{kor-arit}, it suffices to prove Proposition \ref{per-max-exist} for products of negatively curved finite volume locally symmetric spaces and arithmetic locally symmetric spaces. By Proposition \ref{flat-product}, it suffices to show that periodic maximal flats exist in every one of the factors. It is well-known that every negatively curved locally symmetric space contains a closed non-trivial geodesic, i.e.~a periodic maximal flat. For the arithmetic factors, the claim is due to Prasad and Raghunathan \cite[Theorem 2.8]{Prasad-Raghunathan}.
\end{proof}

Until the end of this section we suppose:
\begin{itemize}
\item[(*)] $G$ is a connected semisimple algebraic group defined over $\BQ$, $S=G_\BR^0/K$ is the associated symmetric space, $\Gamma\subset G_\BZ$ is a finite index torsion free subgroup, $F\subset S$ is a maximal flat with associated maximal $\BR$-split torus $A$, and finally 
$$\Lambda\subset\Gamma\cap\CZ_{G_\BR^0}(A_\BR^0)$$
is an abelian subgroup of $\Gamma$ centralizing $A_\BR^0$. 
\end{itemize}
Observe that $\Lambda$ is discrete and torsion free. We denote by $\rank_\BZ(\Lambda)$ the rank of $\Lambda$ as an abelian group, i.e.~the minimal number of elements needed to generate $\Lambda$.

\begin{lem}\label{lem:bound}
The identity component $\Span(\Lambda)$ of the Z-closure of $\Lambda$ is a $\BQ$-torus, and the identity component of $A\cap\Span(\Lambda)$ is an $\BR$-split torus of dimension at least $\rank_\BZ(\Lambda)$.
\end{lem}
\begin{proof}
By Proposition \ref{split-flat}, $\Lambda$ consists of semisimple elements. In particular, it follows from Proposition \ref{prop:split} that $\Span(\Lambda)$ is a torus. Since $\Lambda\subset\Gamma\subset G_\BZ$, it follows that $\Span(\Lambda)$ is $\Gal(\BC/\BQ)$-invariant and hence defined over $\BQ$. Before moving on, observe that since the Z-closure of $\Lambda$ is an algebraic group, it has finitely many connected components. In particular $\Lambda'=\Lambda\cap\Span(\Lambda)$ is a finite index subgroup of $\Lambda$. It follows that $\rank_\BZ(\Lambda')=\rank_\BZ(\Lambda)$ because $\Lambda$ is torsion free.

By assumption $\Lambda$, and hence $\Lambda'$, centralizes $A$. By Proposition \ref{split-flat}, the centralizer $\CZ_{G_\BR^0}(A_\BR^0)$ of $A_\BR^0$ in $G_\BR^0$ is the product of $A_\BR^0$ and a compact group. Recall that the projection 
$$\pol:\CZ_{G_\BR^0}(A_\BR^0)\to A_\BR^0$$
associated to this splitting of $\CZ_{G_\BR^0}(A_\BR^0)$ maps every $g$ to its polar part $\pol(g)$. By the last claim of Proposition \ref{split-flat}, we have that $\pol(g)\in\Span(\Lambda)$ for every $g\in\Lambda'$. In particular $\Span(\pol(\Lambda'))$, the identity component of the Zariski closure of $\pol(\Lambda')=\{\pol(g)\vert g\in\Lambda'\}$, is contained in both $\Span(\Lambda)$ and $A$. Being a connected subgroup of $A$, we deduce from Proposition \ref{prop:split} that $\Span(\pol(\Lambda'))$ is an $\BR$-split torus. Observe that the dimension of $\Span(\pol(\Lambda'))_\BR^0$ as a real Lie group bounds the dimension of $A\cap\Span(\Lambda)$ from below.

Since the kernel of the projection $\pol$ is compact, we deduce that $\pol(\Lambda')$ is a discrete group isomorphic to $\Lambda'$. On the other hand, being a subgroup of $A_\BR^0$, the group $\Span(\pol(\Lambda'))_\BR^0$ acts on $F$ by translations. Since $\Span(\pol(\Lambda'))_\BR^0$ contains the discrete group $\pol(\Lambda')$ we obtain that
\begin{align*}
\dim(A\cap\Span(\Lambda))
&\ge \dim_\BR(\Span(\pol(\Lambda'))_\BR^0\ge\rank_\BZ(\pol(\Lambda'))\\
&=\rank_\BZ(\Lambda')=\rank_\BZ(\Lambda)
\end{align*}
as claimed.
\end{proof}

Suppose now that $\Stab_\Gamma(F)$, the stabilizer of $F$ in $\Gamma$, acts cocompactly on $F$. By Proposition \ref{split-flat}, $\Stab_\Gamma(F)$ is a subgroup of $\CN_{G_\BR^0}(A_\BR^0)$. Since $\CZ_{\CG_\BR^0}(A_\BR^0)$ 
\begin{itemize}
\item has finite index in $\CN_{G_\BR^0}(A_\BR^0)$, and 
\item is the direct product of $A_\BR^0$ and a compact group 
\end{itemize}
we obtain that $\Stab_\Gamma(F)\bs \CN_{G_\BR^0}(A_\BR^0)$ is compact. In particular, for any $g\in G_\BR^0$, the image of the map
$$\CN_{G_\BR^0}(A_\BR^0)\to \Gamma\bs G_\BR^0,\ \ h\mapsto[hg]$$
is compact. It follows from Proposition \ref{q-split-proper} that there is no $\BQ$-split torus whose group of real points normalizes $A_\BR^0$. In particular, we have:

\begin{lem}\label{pre-nosplit}
If $\Stab_\Gamma(F)\bs F$ is compact, then $A\cap\Span(\Lambda)$ does not contain any non-trivial $\BQ$-split torus.\qed
\end{lem}

Still assuming that $\Stab_\Gamma(F)\bs F$ is compact, we bound the dimension of those $\BQ$-split tori commuting with $\Lambda$.

\begin{lem}\label{lem:meat}
Suppose that $\Stab_\Gamma(F)\bs F$ is compact and that $S\subset G$ is a $\BQ$-split torus commuting with $\Lambda$. Then
$$\rank_\BZ(\Lambda)+\dim(S)\le\rank_\BR(G)$$
\end{lem}
\begin{proof}
Since $S$ commutes with $\Lambda$, it also commutes with $\Span(\Lambda)$ and hence with the $\BR$-split torus $A'=A\cap\Span(\Lambda)$. We deduce from Proposition \ref{prop:split} that $\Span(A'\cup S)$, the identity component of the Z-closure of the group generated by $A'\cup S$, is an $\BR$-split torus and hence
\begin{equation}\label{eqlem12}
\dim\Span(A'\cup S)\le\rank_\BR(G)
\end{equation}
We claim that $A'\cap S$ is finite. Otherwise, the identity component of $A'\cap S\subset S$ is a non-trivial $\BQ$-split torus by Proposition \ref{prop:split}. On the other hand, by Lemma \ref{pre-nosplit}, $A'$ does not contain non-trivial $\BQ$-split tori. We have proved that $S$ and $A'$ intersect in a finite set and hence that
\begin{equation}\label{eqlem13}
\dim(A')+\dim(S)\le\dim\Span(A'\cup S)
\end{equation}
The claim follows from \eqref{eqlem12}, \eqref{eqlem13}, and Lemma \ref{lem:bound}.
\end{proof}

The following is the main result of this section:

\begin{prop}\label{meat}
Suppose that $G$ is a connected semisimple algebraic group defined over $\BQ$, $K\subset G_\BR^0$ is a maximal compact subgroup of the identity component of the group of real points, $\Gamma$ is a torsion free finite index subgroup of $G_\BZ$,  $F$ is a flat in the symmetric space $S=G_\BR^0/K$ with $\Stab_\Gamma(F)\bs F$ compact, and $\Lambda\subset\Stab_\Gamma(F)$ is an abelian subgroup of $\Gamma$ stabilizing $F$.

Suppose also that $P$ is a parabolic subgroup defined over $\BQ$ and $S\subset P$ is a $\BQ$-split torus whose centralizer $\CZ_G(S)$ in $G$ is a Levi-subgroup of $P$. If $\Lambda\subset P$ then
$$\rank_\BZ(\Lambda)\le\rank_\BR(G)-\dim(S)$$
\end{prop}
\begin{proof}
Let $A$ be the maximal $\BR$-split torus associated to the maximal flat $F$. Since by Proposition \ref{split-flat} the centralizer $\CZ_{G_\BR^0}(A_\BR^0)$ is a finite index subgroup of $\Stab_{G_\BR^0}(F)$ we deduce that
$$\Lambda'=\Lambda\cap\CZ_{G_\BR^0}(A_\BR^0)$$
has finite index in $\Lambda$. Since $\Lambda$ is torsion free, we have that both $\Lambda$ and $\Lambda'$ have the same rank as abelian groups. 

Observe now that $\Span(\Lambda')$ is contained in $P$. By Lemma \ref{lem:bound}, $\Span(\Lambda')$ is a torus defined over $\BQ$. In particular, $\Span(\Lambda')$ is reductive and hence there is there is $g\in P_\BQ$ with 
$$\Span(\Lambda')\subset g\CZ_G(S)g^{-1}=\CZ_G(gSg^{-1})$$
Replacing the $\BQ$-split torus $S$ by the also $\BQ$-split torus $gSg^{-1}$ we may assume that $\Span(\Lambda')$ was contained in $\CZ_G(S)$ to begin with. In particular, we may assume that $S$ commutes with $\Lambda'$. The claim follows now from Lemma \ref{lem:meat}.
\end{proof}

\section{Proof of Theorem \ref{main}}\label{sec:proof1}

As indicated by the title of this section, we now prove Theorem \ref{main}. We remind the reader of the statement of the theorem:

\begin{named}{Theorem \ref{main}}
In a finite volume locally symmetric manifold, no periodic maximal flat is peripheral. 
\end{named}

Before launching the proof of Theorem \ref{main} we observe the following simple facts used below to simplify our situation:
\begin{itemize}
\item If $\bar M$ is a compact manifold whose interior is homeomorphic to $M$ then a map $f:F\to M$ is peripheral if and only if it is homotopic to a map $F\to\D\bar M$.
\item If $f:F\to M$ is a periodic maximal flat in the locally symmetric manifold $M$ and $\pi:F'\to F$ is a finite cover, then $f\circ\pi:F'\to M$ is again a periodic maximal flat. 
\item If the locally symmetric manifold $M$ contains a peripheral periodic maximal flat and $M'\to M$ is a finite cover, then $M'$ contains a peripheral periodic maximal flat as well.
\end{itemize}

Starting with the proof of Theorem \ref{main}, suppose for the time being that the finite volume locally symmetric manifold $M=\Gamma\bs S$ in consideration is irreducible. The general case will be treated below. 

Also assume for now that $\rank_\BR(M)=1$; equivalently, $M$ is negatively curved, and hence every periodic maximal flat is a non-trivial closed geodesic. Let $\gamma\subset M$ be such a closed geodesic  and denote by $\hat\gamma$ some element in $\pi_1(M)$ in the conjugacy class determined by $\gamma$. Observe that every non-trivial power of $\hat\gamma$ is semisimple. 

Let $\bar S$ be the bordification of $S$ provided by Proposition \ref{thin-thick} and recall that $\bar M=\Gamma\bs\bar S$ is a compact manifold whose interior is homeomorphic to $M$. As observed above, $\gamma$ is peripheral if and only if it can be homotoped within $\bar M$ into $\D\bar M$. Suppose that this is the case. It follows that there is a connected component of $\D\bar S$ which is stabilized by the deck transformation $\hat\gamma$. This contradicts the third claim of Proposition \ref{thin-thick}, namely that the stabilizer of every connected component of $\D\bar S$ contains a finite index subgroup consisting of unipotent elements. This concludes the proof of Theorem \ref{main} if $\rank_\BR(M)=1$.

Suppose from now on that $\rank_\BR(M)\ge 2$ and, seeking a contradiction, suppose that $\phi:F\to M$ is a peripheral periodic maximal cusp. Since we are assuming that $M$ is irreducible, it follows from the arithmeticity theorem that $M=\Gamma\bs S$ is commensurable to an arithmetic locally symmetric manifold. Passing to a suitable finite cover we may assume that:
\begin{enumerate}
\item $M=\Gamma\bs S$ is arithmetic, i.e.
\begin{itemize}
\item$S=G_\BR^0/K$ where $G$ is a connected semisimple algebraic group defined over $\BQ$, and $K$ is a maximal compact subgroup of the identity component $G_\BR^0$ of the group of real points, and
\item $\Gamma$ is a torsion free finite index subgroup of $G_\BZ$.
\end{itemize}
\end{enumerate}
Let $\bar S$ be the bordification of the symmetric space $S$ provided by Theorem \ref{Borel-Serre} and recall that $M$ is homeomorphic to the interior of $\bar M=\Gamma\bs\bar S$. As in the $\rank_\BR(M)=1$ case, the assumption that $\phi:F\to M$ is peripheral implies that:
\begin{itemize}
\item[(2)] There is a map $\psi:F\to\D\bar M$ homotopic within $\bar M$ to the periodic maximal flat $\phi$.
\end{itemize}
It is due to Bieberbach (see \cite{Wolf}) that every closed locally euclidean manifold is finitely covered by a topological torus. Hence, we may make the final assumption that:
\begin{itemize}
\item[(3)] $F$ is a topological torus of dimension $\dim F=\rank_\BR(M)$.
\end{itemize}
We suppose from now on that (1)-(3) are satisfied.

The main tool of the proof of Theorem \ref{main} is the following beautiful result due to McMullen \cite[Corollary 2.2]{McMullen} which we will apply to the topological torus $F$.

\begin{sat}[McMullen]\label{McMullen}
There is no open covering $\BT^n=U_1\cup\dots\cup U_n$ of the $n$-dimensional topological torus $\BT^n$ such that for every component $V$ of $U_i$, the image of the homomorphism
$$H_1(V,\BZ)\to H_1(\BT^n,\BZ)\simeq\BZ^n$$
is an abelian group of rank at most $i-1$.
\end{sat}

In order to prove Theorem \ref{main} we will show that the torus $F$ has a covering of the kind prohibited by McMullen's theorem. This will contradict the assumption that $\phi:F\to M$ is peripheral.
 
\begin{bem}
Suppose that the torus $\BT^n$ is endowed with some simplicial structure. By an open simplex of dimension $k$ we mean a $k$-simplex without its boundary. Abusing notation we will say that $X\subset\BT^n$ is {\em simplicial} if it is a union of open simplices of possibly different dimensions. For example, a closed simplex is a simplicial set. Also, if $\phi:\BT^n\to\Delta$ is a simplicial map to some simplicial complex, then the preimage under $\phi$ of any simplicial set in $\Delta$ is a simplicial set. Refining the simplicial structure of $\BT^n$ and taking open stars, it is easy to see that every simplicial set $X$ is the deformation retract of an open set. In particular, Theorem \ref{McMullen} remains true if the sets $U_i$ are supposed to be simplical instead of open. 
\end{bem}

We continue with the proof of Theorem \ref{main}. Let $\tilde F$ be the universal cover of $F$ and lift the isometric immersion $\phi:F\to M=\Gamma\bs S$ to an isometric embedding $\tilde\phi:\tilde F\to S$ which is $\pi_1(\phi)$-equivariant. Let $\Phi$ be the image of $\pi_1(F)$ under $\pi_1(\phi)$ and observe that $\Phi$ is a free abelian group. Identifying $\tilde F$ with its image under $\tilde\phi$ we have that $\tilde F$ is a maximal flat in the symmetric space $S$ which is $\Phi$-invariant; in particular $\Stab_\Gamma(\tilde F)\bs\tilde F$ is compact.

Recall that $\bar S$ is the Borel-Serre bordification of $S$ and that $\bar M=\Gamma\bs\bar S$. We lift the homotopy between $\phi:F\to M$ and $\psi:F\to\D\bar M$ to $\bar S$ and obtain a $\Phi$-equivariant map
$$\tilde\psi:\tilde F\to\D\bar S$$
By Theorem \ref{Borel-Serre} there is a $\Gamma$-equivariant, and a fortiori $\Phi$-equivariant, map 
$$\tau:\D\bar S\to\Delta_\BQ(G)$$
where $\Delta_\BQ(G)$ is the rational Tits building of $G$. Recall that the latter is a simplicial complex of dimension $\rank_\BQ(G)-1=\rank_\BQ(M)-1$.

We consider the composition of the lift $\tilde\psi$ and the projection $\tau$ and, choosing a sufficiently fine $\Phi$-equivariant simplicial structure of $\tilde F$ and $\Delta_\BQ(G)$, we homotope $\tau\circ\tilde\psi$ to a $\Phi$-equivariant simplicial map 
$$\sigma:\tilde F\to\Delta_\BQ(G)$$
We may assume that the simplicial structure of $\Delta_\BQ(G)$ is a refinement of the standard simplicial structure.

For $i=1,\dots,\rank_\BQ(M)-1$ we denote by $\tilde U_i$ the preimage under $\sigma$ of the union of all $i$-dimensional open simplices in the standard simplicial structure of $\Delta_\BQ(G)$. Observe that the sets $\tilde U_1,\dots,\tilde U_{\rank_\BQ(M)-1}$ are simplicial, $\Phi$-invariant and cover $\tilde F$. We denote by $U_i$ the projection of $\tilde U_i$ to $F$. We claim:

\begin{lem}\label{lem:key}
The image of the homomorphism $H_1(V,\BZ)\to H_1(F,\BZ)$ has at most rank $\rank_\BR(M)-i-1$
for every connected component $V$ of $U_i$.
\end{lem}

It follows directly from Lemma \ref{lem:key} that reordering the cover $U_i$ and adding $\rank_\BR(M)-\rank_\BQ(M)$ copies of the empty set we obtain a simplicial cover
$$F=\emptyset\cup\dots\cup\emptyset\cup U_{\rank_\BQ(M)-1}\cup\dots\cup U_0$$
of the topological torus $F$ of the kind ruled out by McMullen's theorem. This contradicts the assumption that the periodic maximal flat $\phi:F\to M$ is peripheral. In order to conclude the proof of Theorem \ref{main} for irreducible locally symmetric space it remains to prove Lemma \ref{lem:key}.

\begin{proof}[Proof of Lemma \ref{lem:key}]
Suppose that $V$ is a connected component of $U_i$ and let $\tilde V$ be a connected component of the preimage of $V$ under the covering $\tilde F\to F$. Via the isomorphism $\Phi\simeq\pi_1(F)\simeq H_1(F,\BZ)$ we can identify the image of the homomorphism
$$H_1(V,\BZ)\to H_1(F,\BZ)$$
with the subgroup $\Lambda$ of $\Phi$ stabilizing $\tilde V$
$$\Lambda=\Stab_\Phi(\tilde V)=\{\gamma\in\Phi\vert\gamma\tilde V=\tilde V\}$$
Recall that $\tilde V$ is a connected component of the preimage under the map $\sigma$ of the union of all open $i$-dimensional simplices. In particular, there is an open simplex $s\subset\Delta_\BQ(G)$ with 
$$\dim(s)=i$$
and with $\tilde V\subset\sigma^{-1}(s)$.

The equivariance of the map $\sigma$ implies that for all $g\in\Lambda$ we have $gs\cap s\neq\emptyset$. Since the action of $G_\BQ$, and a fortiori $\Lambda$, on $\Delta_\BQ(G)$ is simplicial we deduce that in fact $g\in\Stab_{G_\BQ}(s)$. Recall that by Theorem \ref{facts-tits}, the stabilizer of $s$ in $G_\BQ$ is the intersection of $G_\BQ$ with some parabolic group $P$. We can summarize the preceding discussion with the formula
$$\hbox{Image}\{H_1(V,\BZ)\to H_1(F,\BZ)\}\simeq\Lambda\subset P$$
Since $P_\BQ$ is the stabilizer of the simplex $s\subset\Delta_\BQ(G)$ we obtain from Theorem \ref{facts-tits} that $P$ contains a $\BQ$-split torus $S$ of dimension
$$\dim(S)=\dim(s)+1=i+1$$
whose centralizer $\CZ_G(S)$ is a Levi subgroup of $P$. The claim follows now directly from Proposition \ref{meat}.
\end{proof}

So far we have proved Theorem \ref{main} for irreducible $M$. We sketch now the argument for general finite volume locally symmetric spaces. Observe first that, passing to finite covers, it suffices by Corollary \ref{kor-arit} to prove Theorem \ref{main} for
$$M=\Gamma\bs S=M_1\times\dots\times M_s$$
with $M_1,\dots,M_r$ arithmetic and $M_{r+1},\dots,M_s$ negatively curved.

For each $i$ let $S_i=(G_i)_\BR^0/K_i$ be the universal cover of $M_i$ where $G_i$ is a connected semisimple algebraic group. Let also $\Gamma_i\subset\Isom(S_i)$ be a torsion free lattice with $M_i=\Gamma_i\bs S_i$. By assumption we have $\Gamma_i\subset(G_i)_\BZ$ if $M_i$ is not negatively curved. If $M_i$ is negatively curved we may pass to a finite cover and assume that $\Gamma_i\subset(G_i)_\BR^0$. Before going any further observe that
$$\Gamma=\Gamma_1\times\dots\times\Gamma_s$$
Let $\CP$ be the set of those proper maximal parabolic subgroups of $G=G_1\times\dots\times G_s$ for which there is $i\in\{1,\dots,s\}$ such that:
\begin{itemize}
\item $G_j=\pi_j(P)$ for all $j\neq i$; here $\pi_j$ is the projection of $G$ to the j-th factor $G_j$.
\item If $M_i$ is arithmetic, then $\pi_i(P)$ is defined over $\BQ$.
\item If $M_i$ is negatively curved, then $\pi_i(P)\cap\Gamma_i$ contains a non-trivial unipotent element.
\end{itemize}
Observe that in all cases, the projection $\pi_i(P)$ of $P\in\CP$ to the factor $G_i$ is either the whole group or a proper maximal parabolic subgroup.

We define the {\em rational Tits building} of $M$ to be the simplicial complex $\Delta_\BQ(M)$ with vertex set $\CP$ where $\{P_0,\dots,P_k\}\subset\CP$ determines a $k$-simplex if $P_0\cap\dots\cap P_k$ is parabolic. 

\begin{bem}
Observe that if $M$ were arithmetic to begin with, then the rational Tits buildings $\Delta_\BQ(M)$ and $\Delta_\BQ(G)$ of  $M$ and $G$ respectively, would be the same.
\end{bem}

The lattice $\Gamma$ acts on $\Delta_\BQ(M)$ via its action on $\CP$ by conjugation. We summarize a few facts about $\Delta_\BQ(M)$:
\begin{enumerate}
\item For $i\in\{r+1,\dots,s\}$ denote by $\Delta_\BQ(M_i)$ the discrete set consisting of those maximal parabolic subgroups $P$ in $G_i$ such that $\Gamma_i\cap P$ contains a unipotent element. The rational Tits building $\Delta_\BQ(M)$ is $\Gamma$-equivariantly homeomorphic to the join $\Delta_\BQ(G_1)*\dots*\Delta_\BQ(G_r)*\Delta_\BQ(M_{r+1})*\dots*\Delta_\BQ(M_s)$ of the rational Tits buildings associated to all the factors.
\item If $\{P_0,\dots,P_k\}\subset\CP$ determines a $k$-simplex $\sigma$ in $\Delta_\BQ(M)$, then $\Stab_\Gamma(\sigma)=\Gamma\cap P_0\cap\dots\cap P_k$. 
\item If $\sigma$ is a $k$-simplex in $\Delta_\BQ(M)$ and $\Lambda\subset\Stab_\Gamma(\sigma)$ is an abelian group of $\Gamma$ stabilizing $\sigma$ and consisting of semisimple elements, then $\rank_\BZ(\Lambda)\le\rank_\BR(M)-k-1$. 
\item Let $\bar S_i$ be for $i=1,\dots,s$ the bordification of $S_i$ given by Proposition \ref{thin-thick} and Theorem \ref{Borel-Serre}, and set
$$\bar S=\bar S_1\times\dots\times\bar S_s$$
The manifold $\bar S$ is a $\Gamma$-equivariant bordification of the universal cover $S$ of $M$, and there is a $\Gamma$-equivariant map 
$$\tau:\D\bar S\to\Delta_\BQ(M)$$
Moreover, the quotient $\Gamma\bs\bar S$ is compact, and its interior is homeomorphic to $M$. 
\end{enumerate}
A few remarks on the proof of the validity of (1)-(4). The first claim follows directly from the definition of $\Delta_\BQ(M)$. The second assertion is true because it holds for all the factors of the decomposition of $\Gamma$ as a product of $\Gamma_1,\dots,\Gamma_s$. For irreducible $M$, the third claim is essentially the statement of Proposition \ref{meat}. The validity of (3) follows then again because it is true for all the factors in the product decomposition of $\Gamma$. Finally, the map in (4) is obtained from the maps $\D\bar S_i\to\Delta_\BQ(M_i)$ given by Theorem \ref{Borel-Serre} and the homeomorphism in (1). Details are left to the reader.

In order to conclude the proof of Theorem \ref{main} we just need to remark that given (1)-(4) above, we can repeat word-by-word the argument given in the arithmetic case. \qed
\medskip

\begin{bem}
Before devoting our time to other topics, we would like to observe that $\Delta_\BQ(M)$ is a spherical building of dimension $\rank_\BQ(M)-1$ because it is a join of the spherical buildings corresponding to all the factors. Moreover, the map $\tau$ given by (4) is a homotopy equivalence. This is consistent with the fact that $\D\bar S$ is homotopy equivalent to a bouquet of $(\rank_\BQ(M)-1)$-dimensional spheres. Compare with the discussion at the end of section \ref{sec:peripheral}.
\end{bem}

In some way, one could summarize the proof of Theorem \ref{main} as follows: since only a certain part of the fundamental group of a periodic maximal flat can lie in a parabolic subgroup, periodic maximal flats are not peripheral. On the other hand, a periodic flat
$$\phi:F\to M$$
such that the image of $\pi_1(F)$ under $\pi_1(\phi)$ lies in a proper parabolic $\BQ$-subgroup is certainly peripheral. 

With the same notation as before, suppose that $M=\Gamma\bs G_\BR^0/K$ is a non-compact arithmetic locally symmetric manifold and let $P\subset G$ be a maximal (proper) parabolic subgroup defined over $\BQ$. The group $P$ contains a connected semisimple subgroup $H$ defined over $\BQ$ and with 
$$\rank_\BR(H)=\rank_\BR(G)-1$$
See for instance \cite[11.7]{Borel-arit}. It follows from the Prasad-Raghunathan theorem \cite[Theorem 2.8]{Prasad-Raghunathan} that $H_\BZ$ contains a torsion free abelian subgroup $\Lambda$ of $\rank_\BZ(\Lambda)=\rank_\BR(H)$ and consisting of semisimple elements. Since $\Gamma$ has finite index in $G_\BZ$ we may assume, up to replacing $\Lambda$ by a finite index subgroup, that $\Lambda\subset H_\BZ\cap\Gamma$. The action of $\Lambda$ on the symmetric space $G_\BR^0/K$ stabilizes some flat $F$ of dimension 
$$\dim(F)=\rank_\BR(G)-1$$
Moreover, $\Lambda\bs F$ is compact and hence the map
$$\phi:\Lambda\bs F\to\Gamma\bs G_\BR^0/K=M$$
is a periodic flat. On the other hand, $\Lambda$ is by construction contained in $H_\BZ$ and hence in the rational parabolic subgroup $P$. It follows that $\phi$ is peripheral. We have proved that every non-compact arithmetic locally symmetric manifold contains a peripheral periodic flat of dimension $\rank_\BR(M)-1$. 

For general finite volume locally symmetric manifolds the same statement follows from the arithmetic case, the trivial $\rank_\BR=1$ case, and Corollary \ref{kor-arit}; details are again left to the reader. With the collaboration of the reader we have proved:

\begin{named}{Proposition \ref{optimal}}
For every non-compact locally symmetric manifold $M$ with finite volume there is a peripheral periodic flat $f:F\to M$ with $\dim(F)=\rank_\BR(M)-1$.\qed
\end{named}

Before concluding this section, we would like to observe that Proposition \ref{optimal} shows that in some sense Theorem \ref{main} is optimal. We do not know if Theorem \ref{peripheral} is also optimal but we suspect that this is the case. More precisely, we believe that the following question should have a positive answer:

\begin{quest}
Let $M$ be a finite volume locally symmetric space. Is there a non-peripheral periodic flat
$$\phi:F\to M$$ 
with $\dim(F)=\rank_\BQ(M)$?
\end{quest}

\section{Remarks}\label{sec:final}
Recall that a compact CW-complex $\CX$ in the locally symmetric space $M$ is a {\em spine} if $\CX$ is a deformation retract of $M$. A spine $\CX$ is a {\em soul} of $M$ if there is a homotopy 
$$H:[0,1]\times(M\setminus\CX)\to\bar M$$
starting at the identity and with $H(\{1\}\times(M\setminus\CX))\subset\D\bar M$. Here $\bar M$ is a compact manifold whose interior is homeomorphic to $M$. Notice that the complement in $\bar M$ of a sufficiently small open neighborhood of $\D\bar M$ is a soul of $M$. 

It follows directly from the definition that every map whose image misses some soul is peripheral and hence we deduce from Theorem \ref{main}:

\begin{named}{Corollary \ref{kor:soul}}
If $\CX\subset M$ is a soul of $M$, then the image of every map homotopic to a periodic maximal flat intersects $\CX$.\qed
\end{named}

At this point we would like to ask if Corollary \ref{kor:soul} is also true for general spines:

\begin{quest}
Let $\CX\subset M$ be a spine. Is there some periodic maximal flat $\phi:F\to M$ which can be homotoped outside of $\CX$?
\end{quest}

Before moving on we get a concrete incarnation of Corollary \ref{kor:soul}:

\begin{kor}\label{orbit}
Let $G$ be a connected semisimple algebraic group, $K\subset G_\BR^0$ a maximal compact subgroup of the group of real points, $\Gamma\subset G$ a torsion free lattice, and $\CX$ a soul of the locally symmetric space $M=\Gamma\bs G_\BR^0/K$. If $A\subset G$ is a maximal $\BR$-split torus with $(A_\BR^0\cap\Gamma)\bs A_\BR^0$ compact, then for every $g\in G$ the image of the map
$$A_\BR^0\to M,\ \ h\mapsto[hg]$$
intersects $\CX$. Here $[x]$ is the projection to $M$ of $x\in G_\BR^0$.
\end{kor}
\begin{proof}
Let $F$ be the maximal flat in $S=G_\BR^0/K$ invariant under $A_\BR^0$ and choose $g_0\in G_\BR^0$ with 
$$F=\{[hg_0]\vert h\in A_\BR^0\}$$
Given any path $(g_t)_{t\in[0,1]}$ in $G$ starting at $g_0$ and ending in $g_1=g$ we consider the map
$$\tilde H: A_\BR^0\times[0,1]\to S,\ \ (h,t)\mapsto H_t(h)=[hg_t]$$
The map $\tilde H$ is $A_\BR^0$-equivariant and hence descends to a homotopy
$$H:\left((A_\BR^0\cap\Gamma)\bs A_\BR^0\right)\times[0,1]\to M$$
Identifying the flat $F$ with $A_\BR^0$ we can consider $H_0$ as a periodic maximal flat. By Corollary \ref{kor:soul}, the image of $H_1$, the final map in the homotopy, intersects $\CX$. Since the image of the map in the statement is equal to the image of $H_1$, the claim follows.
\end{proof}

Suppose from now on that $M=\Gamma\bs\SL_\BR/\SO_n$ where $\Gamma$ is a finite index torsion free subgroup of $\SL_n\BZ$. To every $A\in\SL_n\BR$ we associate the lattice $A^{-1}\BZ^n\subset\BR^n$. Doing so, we identify the symmetric space $\SL_n\BR/\SO_n$ with the space of isometry classes of marked lattices in $\BR^n$ with covolume $1$.

\begin{bem}
It would be more natural to identify the space of isometry classes of marked lattices in $\BR^n$ with $\SO_n\bs\SL_n\BR$. However, doing so we would have had to resign ouserselves to work with right actions throughout the paper.
\end{bem}

As mentioned in the introduction, a lattice is {\em well-rounded} if the set of its shortest non-trivial vectors spans $\BR^n$ as an $\BR$-vector space. The subset $\CX\subset\SL_n\BR/\SO_n$ consisting of isometry classes of well-rounded lattices is the so-called {\em well-rounded retract}; see \cite{Ash,wr-minimal,no-spine,Soule} for assorted properties of the well-rounded retract. It is clear from the construction that $\CX$ is $\SL_n\BZ$-invariant; abusing notation, denote also by $\CX$ the projection of $\CX$ to $M$ and recall that it is due to Soul\'e and Ash that $\CX$ is a spine of $M$. As mentioned in the introduction, it follows from \cite[Proposition 6.3]{wr-minimal} that $\CX$ is also a soul.

\begin{prop}\label{wr-soul}
The projection of the well-rounded retract is a soul of $M$.\qed
\end{prop}

Denote now by $H$ the subgroup of $\SL_n\BR$ consisting of diagonal matrices whose entries along the diagonal are positive. Observe that $H$ is the identity component of the set of real points of a maximal $\BR$-split torus in $\SL_n\BC$. Suppose that $A\in\SL_n\BR$ is such that the projection of $HA=\{BA\vert B\in H\}$ to $\SL_n\BR/\SL_n\BZ$ is compact; equivalently the projection of $A^{-1}H$ to $\SL_n\BZ\bs\SL_n\BR$ is compact. 

As discussed in \cite[Section 3]{McMullen}, the assumption that the projection of $HA$ to $\SL_n\BR/\SL_n\BZ$ is compact implies that $(A^{-1}HA)\cap\SL_n\BZ$, and hence $(A^{-1}HA)\cap\Gamma$, is a cocompact subgroup of the maximal $\BR$-split torus $A^{-1}HA$. It follows from Corollary \ref{orbit} that the projection of $A^{-1}H=(A^{-1}HA)A^{-1}$ to $M=\Gamma\bs\SL_n\BR/\SO_n$ intersects the well-rounded retract $\CX$ of $M$. Equivalently, $A^{-1}H$ intersects the locus of those $C\in\SL_n\BR$ with $C\BZ^n$ well-rounded. In other words, there is $B\in H$ with $A^{-1}B\in\CX$. This means that the lattice $B^{-1}A\BZ^n$ is well-rounded. Since obviously $B^{-1}\in H$, we have proved:

\begin{named}{Theorem \ref{McMullen-wr}}[McMullen \cite{McMullen}]
Let $H\subset\SL_n\BR$ be the diagonal group and suppose that $A\in\SL_n\BR$ is such that the projection of $HA$ is compact in $\SL_n\BR/\SL_n\BZ$. Then there is $B\in H$ such that $BA\BZ^n$ is a well-rounded lattice.\qed
\end{named}

McMullen proved in fact something stronger. Namely, he proved that whenever the projection of $HA$ to $\SL_n\BR/\SL_n\BZ$ is precompact, then there is $B$ in the closure $\overline{HA\SL_n\BZ}$ of the (right) $\SL_n\BZ$-orbit of $HA$ with $B\BZ^n$ well-rounded. Equivalently, if the projection of $A^{-1}H$ to $M$ is bounded then $\overline{A^{-1}H}$ intersects the well-rounded retract.  On the other hand, it is due to Tomanov-Weiss \cite{Tomanov-Weiss} that every locally symmetric space $M$ contains a compact set $C$ which intersects the image of every maximal flat in $M$. This raises the following question:

\begin{quest}\label{ques2}
Let $M=\Gamma\bs S$ be a finite volume locally symmetric manifold. Is there a compact set $C\subset M$ such that no maximal flat $F\subset S$ is bounded homotopic to some $F'$ whose projection to $M$ is disjoint from $C$?
\end{quest}

Here we say that a homotopy is {\em bounded} if the lengths of the trajectories are uniformly bounded above by some $L<\infty$. We would like to observe that if we replace bounded homotopic by proper homotopic, then the answer to 
\begin{figure}[h]
         \centering
         \includegraphics[width=8cm]{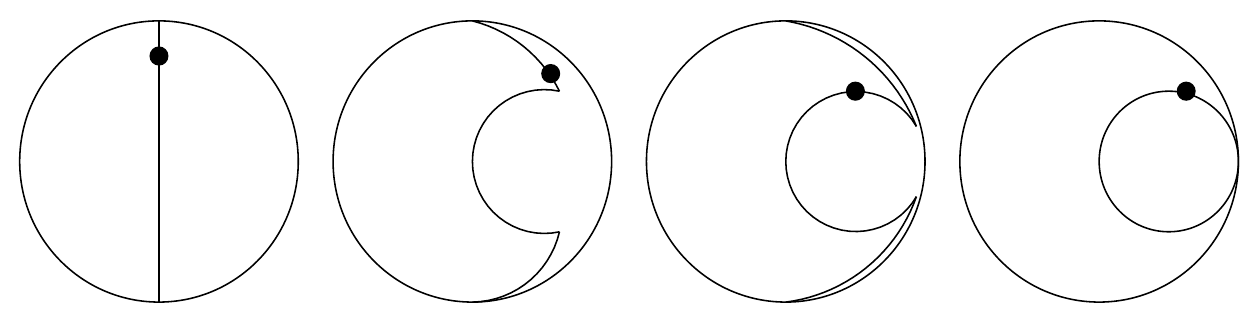}
         \caption{A proper homotopy between a geodesic and a horosphere in $\BH^2$.}\label{fig}
\end{figure}
Question \ref{ques2} is negative even for non-compact hyperbolic surfaces; see figure \ref{fig}.

\bigskip

{\small 
\noindent Alexandra Pettet, Department of Mathematics, University of Michigan

\noindent \texttt{apettet@umich.edu}

\bigskip

\noindent Juan Souto, Department of Mathematics, University of Michigan

\noindent \texttt{jsouto@umich.edu}}

\end{document}